\documentclass[12pt, reqno]{amsart}

\usepackage{amssymb, amsmath, amsthm, mathrsfs, verbatim}
\usepackage[backref]{hyperref}
\usepackage{graphicx}
\usepackage{epstopdf,epsfig}
\usepackage[author-year,backrefs]{amsrefs}
\usepackage{amscd}   % for commutative diagrams
\usepackage{fullpage}
\usepackage[all]{xy} % for complicated commutative diagrams
\usepackage[section]{placeins}
\usepackage[usenames,dvipsnames]{color}
\newcommand{\change}[1]{{\color{blue} #1}}
\newcommand{\changeK}[1]{{\color{blue} #1}}

\newcommand{\Op}{{O}_p}
\newcommand{\cyclicp}{{C}_p}

\newcommand{\Cm}{{C}_m}
\newtheorem{lemma}{Lemma}[section]
\newtheorem{theorem}[lemma]{Theorem}

\newtheorem{prop}[lemma]{Proposition}

\newtheorem{claim*}{Claim}

\newtheorem{defn}[lemma]{Definition}

\newtheorem{proposition}[lemma]{Proposition}

\newtheorem{definition}[lemma]{Definition}

\newtheorem*{theorem*}{Theorem}

\theoremstyle{plain}
\theoremstyle{remark}
\newtheorem{rema}[lemma]{Remark}

\newtheorem{remark}[lemma]{Remark}

\newcommand{\Aff}{{\mathbb A}}

\newcommand{\PP}{{\mathbb P}}

\newcommand{\F}{{\mathbb F}}
\newcommand{\Q}{{\mathbb Q}}

\newcommand{\Z}{{\mathbb Z}}

\newcommand{\calC}{{\mathcal C}}

\newcommand{\calO}{{\mathcal O}}

\newcommand{\calS}{{\mathcal S}}

\newcommand{\calZ}{{\mathcal Z}}
\newcommand{\OO}{{\mathcal O}}

\newcommand{\pp}{{\mathfrak p}}
\newcommand{\mm}{{\mathfrak m}}

\newcommand{\prob}{{\operatorname{Prob}}}

\newcommand{\A}{{\mathbb A}}

\newcommand{\cO}{\mathcal{O}}

\newcommand{\scrO}{\mathscr{O}}
\newcommand{\scrP}{\mathscr{P}}
\newcommand{\scrX}{\mathscr{X}}
\newcommand{\scrV}{\mathscr{V}}

\newcommand{\scrI}{\mathscr{I}}

\newcommand{\cC}{{\mathcal C}}

\newcommand{\fm}{\mathfrak m}
\newcommand{\fp}{\mathfrak p}

\newcommand{\ffp}{{\mathbb F}_p}

\DeclareMathOperator{\lcm}{lcm}

\DeclareMathOperator{\Char}{char}

\DeclareMathOperator{\Spec}{Spec}

\numberwithin{equation}{section}
\numberwithin{table}{section}

\newcommand{\defi}[1]{\textsf{#1}} % for defined terms

\title{Evidence for the dynamical Brauer-Manin criterion}
\thanks{
E.A. was partially supported by AG Laboratory NRU-HSE, RF government grant, ag. 11.G34.31.0023. Her research was carried out within the NRU HSE
Academic Fund Program for 2013-2014, research grant No. 12-01-0107. 
P.K. was partially supported by grants from the G\"oran Gustafsson
Foundation and the Swedish Research Council. K.N. was partially supported
by NSF grant DMS-0901149.
B.V. was supported by ICERM, the Centre Interfacultaire Bernoulli, and NSF grant DMS-1002933.
J.F.V. was supported by the Simons Foundation (grant \#234591) and by the Centre Interfacultaire Bernoulli.}

\author{Ekaterina Amerik}
\address{Laboratory of Algebraic Geometry, 
National Research University Higher School of Economics, 
7 Vavilova Str., Moscow, Russia, 117312; and Laboratoire de Math., Universit\'e Paris-Sud,
Campus d'Orsay, B\^atiment 425, 91405 Orsay, France}
\email{Ekaterina.Amerik@math.u-psud.fr}
\urladdr{www.math.u-psud.fr/\~{}amerik}

\author{P\"ar Kurlberg}
\address{Department of Mathematics, KTH Royal Institute of Technology,
SE-100 44 Stockholm, Sweden}
\email{kurlberg@math.kth.se}
\urladdr{www.math.kth.se/\~{ }kurlberg}

\author{Khoa Nguyen}
\address{
Department of Mathematics,
University of California,
Berkeley, CA 94720 
}

\email{khoanguyen2511@gmail.com}
\urladdr{http://math.berkeley.edu/\~{}khoa}

\author{Adam Towsley}
\address{Department of Mathematics, CUNY Graduate Center, New York, NY, 10016, USA}
\email{atowsley@gc.cuny.edu}
\urladdr{https://sites.google.com/site/atowsley42/}

\author{Bianca Viray}
\address{Department of Mathematics, Box 1917, Brown University, Providence, RI
			02912, USA}
\email{bviray@math.brown.edu}
\urladdr{http://math.brown.edu/\~{}bviray}

\author{Jos\'e Felipe Voloch}
\address{Department of Mathematics, University of Texas, Austin, TX 78712, USA}
\email{voloch@math.utexas.edu}
\urladdr{http://www.ma.utexas.edu/users/voloch/}

\date{}

\begin{document}

	%%%%%%%%%%%%%%%%%%%%%%%%%%%%%%%%%%%%%%%%%%%%%%%%%%%%%%%%%%%%%%%%%%%%%%%%%%%%
	\begin{abstract}
		Let $\varphi\colon X \to X$ be a morphism of a variety over a number field $K$.  We consider local conditions and a ``Brauer-Manin'' condition, defined by Hsia and Silverman, for the orbit of a point $P\in X(K)$ to be disjoint from a subvariety $V\subseteq X$, i.e., for $\OO_{\varphi}(P)\cap V = \emptyset.$  We provide evidence that the dynamical Brauer-Manin condition is sufficient to explain the lack of points in the intersection $\OO_{\varphi}(P)\cap V$; this evidence stems from a probabilistic argument as well as unconditional results in the case of \'etale maps.
	\end{abstract} 
	%%%%%%%%%%%%%%%%%%%%%%%%%%%%%%%%%%%%%%%%%%%%%%%%%%%%%%%%%%%%%%%%%%%%%%%%%%%%

	\maketitle
%%%%%%%%%%%%%%%%%%%%%%%%%%%%%%%%%%%%%%%%%%%%%%%%%%%%%%%%%%%%%%%%%%%%%%%%%%%%%%%%
\section{Introduction}\label{sec:intro}%%%%%%%%%%%%%%%%%%%%%%%%%%%%%%%%%%%%%%%%%
%%%%%%%%%%%%%%%%%%%%%%%%%%%%%%%%%%%%%%%%%%%%%%%%%%%%%%%%%%%%%%%%%%%%%%%%%%%%%%%%

	In recent work, Hsia and Silverman~\cite{hsiasilverman} ask a dynamical question in analogy with a question of Scharaschkin~\cite{scharaschkin}; the dynamical question is as follows.  Let $\varphi\colon X \to X$ be a self-morphism of a variety over a number field $K$, let $V\subseteq X$ be a subvariety, and let $P\in X(K)$.  Hsia and Silverman ask whether the closure of the intersection of the orbit 
	\[
\change{\OO_{\varphi}(P) := \{P, \varphi(P),\varphi(\varphi(P)),\ldots \}}
	\] 
	with the subvariety $V$ is equal to the intersection of $V(\A_K)$ with the closure of $\OO_{\varphi}(P)$ in the adelic topology, i.e. whether
	\[
\change{\calC(V(K)\cap \OO_{\varphi}(P)) = V(\A_K)\cap \calC(\OO_{\varphi}(P)),}
	\]
	where $\calC(-)$ denotes the closure in the adelic topology.
	
	The purpose of this paper is to give some justification to the
assertion that a closely related question has a positive answer. 
In particular, when $K={\mathbb Q}$, \change{(and a choice of
an integral model for $V$ is made)} we give evidence for the assertion that
	\begin{equation}\label{eq:DBM}
		V(K)\cap \OO_{\varphi}(P) = \emptyset \Rightarrow
		V(\Z/m\Z)\cap (\OO_{\varphi}(P)\bmod{m}) = \emptyset
	\end{equation}
	for some integer $m$, as long as the map $\varphi$ is sufficiently random.

	Our evidence is twofold: a probabilistic argument that holds under certain randomness assumptions, and unconditional results in the case that $\varphi$ is \'etale and $V$ is $\varphi^k$-invariant or $\varphi$-preperiodic.
	
	The probabilistic argument is given in Section~\ref{sec:prob_proof} and is based on a probabilistic argument of Poonen~\cite{poonen} for the original question of Scharaschkin. Precisely, we show the following:
	\begin{theorem}
	Assume that $\varphi$ is sufficiently random (see Section~\ref{subsec:Assumptions} for the precise statement) and that the lengths of the cyclic part of the orbit $\OO_{\varphi}(P)$ modulo $p$
are integers with at least the same probability of being smooth as that of a random integer of similar size. Then there exists a sequence of  \emph{squarefree} integers $m$ such that
	\[
		\prob( \Cm \cap V(\Z/m\Z) = \emptyset  ) = 1 - o(1)
	\]
	as $m \to \infty$.  Here $\Cm$ denotes the cyclic part of $\calO_{\varphi}(P)$ in $X(\Z/m)$.
	\end{theorem}
	\change{
	\begin{remark}
          If the intersection of the orbit $\OO_{\varphi}(P)$ modulo
          $m$ with $V(\Z/m\Z)$ is contained only in the tail of
          $\OO_{\varphi}(P)\bmod m$, then the intersection
          $(\OO_{\varphi}(P) \bmod{m})\cap V(\Z/m\Z)$ consists of
          finitely many \emph{iterates} of $P$ under $\varphi$.  More specifically, $\varphi^n(P)\bmod m\notin V(\Z/m\Z)$ for all $n > N_0$,
          where $\varphi^n$ denotes  the composition of
            $\varphi$ with itself $n$ times.
          Therefore $V(K)\cap \OO_{\varphi}(P)$ is contained in 
		  \[
		  \{P, \varphi(P), \varphi^2(P), \ldots, \varphi^{N_0}(P)\},
		  \]
		  and so can be computed by a finite computation.
          Therefore, it is reasonable to restrict to the cyclic part
          of the orbit in the above theorem.
	\end{remark}}
	% The restriction to the cyclic part is reasonable, since if the intersection of the orbit $\OO_{\varphi}(P)$ modulo $m^k$ with $V(\Z/m^k\Z)$ is contained only in the tail of $\OO_{\varphi}(P)\bmod m^k$ for all $k$ sufficiently large then $\calO_{\varphi}(P) \cap V \ne \emptyset$.  Indeed, if the subvariety intersects only the tail of the orbit mod $m$, then it must be the same iterate that intersects the subvariety mod $m^k$ for all $k>1$, so then that iterate must lie in the subvariety $V$.  	
In Section~\ref{sec:computations}, we provide numerical evidence for the randomness assumptions needed in the heuristic argument from Section~\ref{sec:prob_proof}.  We also describe experiments on randomly generated morphisms of $\Aff^5$ which support the argument that~\eqref{eq:DBM} holds.

	The unconditional results are the focus of Section~\ref{sec:DynamicalHasse}.  Assume that $X$ is quasi-projective, that $\varphi$ is \'etale, and that $\varphi^{k}(V)\subseteq V$, \change{i.e., that $V$ is \emph{$\varphi^{k}$-invariant}}, for some positive integer $k$. Under these assumptions we show that if $\OO_{\varphi}(P)\cap V = \emptyset$, then for all but finitely many primes $p$, there exists an $n = n(p)$ such that $\calO_{\varphi}(P)\bmod{p^n} \cap V(\Z/p^n\Z) = \emptyset.$  If every irreducible component of $V$ is \change{\emph{preperiodic}, i.e., $\varphi^{k_0 + k}(V) = \varphi^{k_0}(V)$ for some integers $k_0,k$ with $k>0$,} $X$ is \changeK{quasi-projective}, and $\varphi$ is \'etale and closed, then we obtain the same result.  More generally, we prove
	\begin{theorem}
			Let $X$ be a quasi-projective variety over a global field $K$. Assume that $\varphi$ is \'etale, and either \change{(1) that $V$ is $\varphi^m$-invariant or (2) that $\varphi$ is closed and every irreducible component of $V$ is $\varphi$-preperiodic.}  If $P\in X(K)$ is such that $V(K)\cap \calO_{\varphi}(P)= \emptyset$ then, for all but finitely many primes $v$,
			\[
				V(K_v)\cap \cC_v(\calO_{\varphi}(P)) = \emptyset,
			\]
			where $\cC_v(-)$ denotes closure in the $v$-adic topology.
	\end{theorem}
	In~\ref{subsec:closing}, we discuss whether the assumption \change{that $V$ is preperiodic or $\varphi^k$-invariant} can be weakened in any way.

	%%%%%%%%%%%%%%%%%%%%%%%%%%%%%%%%%%%%%%%%%%%%%%%%%%%%%%%%%%%%%%%%%%%%%%%%%%%%
	\section*{Acknowledgements}%%%%%%%%%%%%%%%%%%%%%%%%%%%%%%%%%%%%%%%%%%%%%%%%%
	%%%%%%%%%%%%%%%%%%%%%%%%%%%%%%%%%%%%%%%%%%%%%%%%%%%%%%%%%%%%%%%%%%%%%%%%%%%%
		This project was started during the workshop on Global Arithmetic Dynamics, part of the semester program on Complex and Arithmetic Dynamics at the Institute for Computational and Experimental Research Mathematics (ICERM).  We thank Zo\'e Chatzidakis, Alice Medvedev, and Thomas Scanlon for helpful discussions during the workshop, and we are grateful to the ICERM staff and the organizers of the workshop, Xander Faber, Michelle Manes, Lucien Szpiro, Thomas Tucker, and Michael Zieve, for their support.  We thank Serge Cantat for pointing out a construction in Section~\ref{subsec:closing}.

%%%%%%%%%%%%%%%%%%%%%%%%%%%%%%%%%%%%%%%%%%%%%%%%%%%%%%%%%%%%%%%%%%%%%%%%%%%%%%%%
\section{Probabilistic Proof}\label{sec:prob_proof}%%%%%%%%%%%%%%%%%%%%%%%%%%%%%
%%%%%%%%%%%%%%%%%%%%%%%%%%%%%%%%%%%%%%%%%%%%%%%%%%%%%%%%%%%%%%%%%%%%%%%%%%%%%%%%

	Let $X$ be a projective variety over $\Q$, let $V\subset X$ be a \change{closed} $\Q$-subvariety, and let $\varphi\colon X\to X$ be a $\Q$-morphism.  We write
	\[
		d_1 := \dim(X), \quad\textup{ and }\quad d_2 := \dim(V).
	\]
	% For any positive integer $m$, we let $\Om$ denote the reduction of $\calO_{\varphi}(P)$ modulo $m$, and let $\Cm$ denote the cyclic part of $\Om$.

	The goal of this section is to prove, under some randomness
        assumptions, that if $\calO_{\varphi}(P) \cap V = \emptyset$,
        then with probability $1$ there exists a positive integer $m$
        such that\change{, roughly speaking, the cyclic part of
          $\OO_\varphi(P)\bmod{m}$ is disjoint from $V(\Z/m\Z)$}.
\change{To make this precise, we need some further notation. 
  Fix a finite set of primes $S$ such that $X, V$ and $P$ extend to
  flat projective models $\scrX, \scrV$ and $\scrP$, respectively,
  over $\Z_S$.  After possibly enlarging $S$, we also assume that
  $\varphi$ extends to a morphism $\tilde\varphi\colon\scrX\to\scrX$.
  Then for any integer $m$ which is relatively prime to all elements
  of $S$, we may consider the base change $X_m := \scrX\times_{\Z_S}
  \Z/m\Z$. Similarly, we have $P_m := \scrP\times_{\Z_S} \Z/m\Z$, $V_m
  := \scrV\times_{\Z_S} \Z/m\Z$ and $\varphi_m := \tilde\varphi|_{X_m} \colon
  X_m \to X_m$.  We will often abuse notation and write $V(\Z/m)$
  for $V_m(\Z/m)$.
		
		We write
		\[
			O_m := \{ P_m, \varphi_m(P_m), \varphi_m(\varphi_m(P_m)), \ldots \}.
		\]
		As $O_m$ is contained in the finite set $X_m(\Z/m)$,
                there is some pair of non-negative integers $k_0 <
                k_1$ such that $\varphi_m^{k_0}(P_m) =
                \varphi_m^{k_1}(P_m)$.  Let $k_0$ be the minimal such
                integer; then we define the cyclic part of $O_m$ as
		\[
			C_m := \{\varphi_m^{k_0}(P_m), \varphi_m^{k_0 + 1}(P_m), \varphi_m^{k_0 + 2}(P_m), \ldots \}.
		\] 
		
%		Then we make the following assumptions.

		}	

	%%%%%%%%%%%%%%%%%%%%%%%%%%%%%%%%%%%%%%%%%%%%%%%%%%%%%%%%%%%%%%%%%%%%%%%%%%%%
	\subsection{Assumptions on $\varphi$ and $V$}\label{subsec:Assumptions}%%%%%
	%%%%%%%%%%%%%%%%%%%%%%%%%%%%%%%%%%%%%%%%%%%%%%%%%%%%%%%%%%%%%%%%%%%%%%%%%%%%
		With the above notation, our randomness assumptions on
                $\varphi$ and $V$ are then 
                based on the following heuristics:
		\begin{enumerate}
			\item The reduction of morphisms $\varphi\colon X\to X$ modulo $p$ behave like random maps on a finite set $S = X(\F_p)$.
			\item For any $y \in X(\F_p)$, the condition that $y \in O_p$ is independent from the condition that $y\in V(\F_p)$.
			\item \change{For any $y\in X(\Q)\setminus V(\Q)$}, the condition that $\bar{y}\in V(\F_p)$ is independent from the condition that \change{$\bar{y}\in {V}(\F_q)$} for primes $p\neq q$.
		\end{enumerate}
		These heuristics are made precise as follows.  
		\begin{enumerate}
			\item \defi{Orbit size assumption:} For all primes $p\notin S$,
				\begin{equation}\label{eq:A1}
					|\Op| = p^{d_{1}/2+o(1)}, \quad
					\textup{ and } \quad |\cyclicp| = p^{d_{1}/2+o(1)}.
				\end{equation}
			\item \defi{Independence assumption on $\calO_{\varphi}(P)$ and 
				$V$  modulo $p$}: For all primes $p\notin S$
				\begin{equation}
					\begin{array}{rl}
					\phantom{ and }
						\prob( x_{1},\ldots, x_{k} \not \in V(\ffp)\ | 
							\ x_{1},\ldots, x_{k} \in \Op  ) & 
							= (1-1/p^{d_{1}-d_{2}+o(1)} )^{k}, \textup{ and }\\
						\prob( x_{1},\ldots, x_{k} \not \in V(\ffp)\ | 
							\ x_{1},\ldots, x_{k} \in \cyclicp  )  &
							= (1-1/p^{d_{1}-d_{2}+o(1)} )^{k},
					\end{array}
					\label{eq:A2}
				\end{equation}
				for $k = O(p^{d_{1}/2+o(1)})$, and $x_{1}, \ldots, x_{k}$ distinct points modulo $p$.	
			\item \defi{Asymptotic independence modulo (large) primes:} if $\OO_{\varphi}(P)\cap V(\Q) = \emptyset$, then
				\begin{equation}\label{eq:A3}
					\begin{array}{rl}
					\phantom{ and }
						\prob( \Op \cap V(\ffp) \neq \emptyset \ \forall p>T  ) &
							= (1+o(1)) \cdot\prod_{p>T}
							\prob(\Op\cap V(\ffp)\neq \emptyset  ),\\
						\prob( \cyclicp \cap V(\ffp) \neq \emptyset \ \forall p>T)&
							= (1+o(1)) \cdot \prod_{p>T}
							\prob( \cyclicp \cap V(\ffp) \neq \emptyset  ).
  					\end{array}
				\end{equation}
		\end{enumerate}
		\change{\begin{remark}
			Any two different models will differ at only finitely many primes.  Therefore, since assumptions~\eqref{eq:A1}--\eqref{eq:A3} say something about behavior at sufficiently large primes, these assumptions are really assumptions on $\phi,X,V$ and $P$ and not $\tilde\phi,\scrX,\scrV,\scrP$.  
		\end{remark}
		}
		 If $A$ is a (large) finite set, and $f : A \to A$ is a sufficiently random map, then cardinality of the forward orbit of a random starting point is likely to be of size $|A|^{1/2+o(1)}$~\cite{FO-RandomMaps}.  This, together with the Weil conjectures which give that $|X(\ffp)|= p^{d_{1}+o(1)}$ and $|V(\ffp)|
  =p^{d_{2}+o(1)}$, shows that the above heuristics imply the precise assumptions~\eqref{eq:A1}--\eqref{eq:A3}. 
		\begin{remark}\label{rem:AssumptionFailures}
			We warn the reader that there are maps of special type for which the random map heuristic does not apply.  For example, for linear automorphisms on $\PP^n$, orbits can be as large as $p^{d_{1}+o(1)}$. 		
		\end{remark}		

	%%%%%%%%%%%%%%%%%%%%%%%%%%%%%%%%%%%%%%%%%%%%%%%%%%%%%%%%%%%%%%%%%%%%%%%%%%%%
	\subsection{Nonempty intersections modulo $p$}%%%%%%%%%%%%%%%%%%%%%%%%%%%%%%
	%%%%%%%%%%%%%%%%%%%%%%%%%%%%%%%%%%%%%%%%%%%%%%%%%%%%%%%%%%%%%%%%%%%%%%%%%%%%

		Under assumptions~\eqref{eq:A1}--\eqref{eq:A3} and if $V$ has sufficiently small codimension, then, with probability $1$, we will have $\calO_{\varphi}(P)\cap V(\ffp) \neq \emptyset$ for all but finitely many $p$ \emph{even if} $ \calO_{\varphi}(P)$ and $V$ have empty intersection. The intuition is as follows. If $V$ is, say, of codimension $1$ in $X$, then the probability that the reduction of $\varphi^n(P)$ modulo $p$ is contained in $V(\F_p)$ should equal $1/p$. Furthermore, if the orbit is of length $p^{d_{1}/2}$, as we would expect, then the likelihood of $V(\ffp) \cap \Op = \emptyset$ should be given by $(1-1/p)^{|\Op|} = \exp( - |\Op|/p)$, assuming $p$ is sufficiently large.  Thus, if the orbits are long, there will be ``accidental'' intersections modulo $p$.

		\begin{proposition} 
			Assume that assumptions~\eqref{eq:A1}--\eqref{eq:A3} hold, and that $d_{2} > d_{1}/2$.  Then, as $T \to \infty$, 
			\[
				\prob( \Op \cap V(\ffp) \neq \emptyset  \ \forall p>T  ) =
				1-e^{-T^{d_{2}-d_{1}/2+o(1)}}.
			\]
		\end{proposition}
		\begin{remark}\label{rem:big-size-ok}
			The proposition still holds under a relaxed orbit size
  assumption: it is sufficient to assume $|\Op| = p^{d_{1}'/2+o(1)}$ for some $d_1'\in \Z_{\geq0}$ such that $d_{2}>d_{1}'/2$.
		\end{remark}

		\begin{proof}
		%  Given a prime $p$ of good reduction,  let $\Op$ denote the
		%  reduction of $\calO_{\varphi}(P)$ modulo $p$.  
			By assumption~\eqref{eq:A1}, $|\Op| = p^{d_{1}/2+o(1)}$, and by the Weil conjectures, $|V(\ffp)| = p^{d_{2}+o(1)}$, $|X(\ffp)| = p^{d_{1}+o(1)}$.  Therefore, assumption~\eqref{eq:A2} gives that 
			\begin{align*}
				\prob( \Op \cap V(\ffp) = \emptyset  )
				& = \left(1-p^{d_{2}-d_{1}+o(1)}\right)^{|\Op|}
				  = \left(1-p^{d_{2}-d_{1}+o(1)}\right)^{p^{d_{1}/2+o(1)}}\\
				& = e^{-p^{d_{1}/2-(d_{1}-d_{2})+o(1)}}
				  = e^{-p^{d_{2}-d_{1}/2+o(1)}}.
			\end{align*}
			In particular, 
			\[
				\prob( \Op \cap V(\ffp) \neq \emptyset  ) 
				= 1 - e^{-p^{d_{2}-d_{1}/2+o(1)}}, 
			\]
			and therefore, by assumption~\eqref{eq:A3}, \change{
			\[
				\prob( \Op \cap V(\ffp) \neq \emptyset  \ \forall p>{T}  ) 
				= \prod_{p > T} (1-e^{-p^{d_{2}-d_{1}/2+o(1)}} )
				= 1 - e^{-T^{d_{2}-d_{1}/2+o(1)}}.
			\]}
		\end{proof}

	%%%%%%%%%%%%%%%%%%%%%%%%%%%%%%%%%%%%%%%%%%%%%%%%%%%%%%%%%%%%%%%%%%%%%%%%%%%%
	\subsection{Empty intersections for some composite $m$}\label{sec:mod-m}%%
	%%%%%%%%%%%%%%%%%%%%%%%%%%%%%%%%%%%%%%%%%%%%%%%%%%%%%%%%%%%%%%%%%%%%%%%%%%%%

		The situation is quite different over composite integers $m$.  Indeed, the main result of this section is that, assuming some further randomness properties, there exist integers $m$ such that the probability that $C_m$ and $V(\Z/m\Z)$ are disjoint is arbitrarily close to $1$.

		We begin by recalling some background on smooth numbers.  
		\begin{defn}
			An integer $n$ is \defi{$y$-smooth} if all primes $p$ dividing $n$ are bounded above by $y$.  Define
			\[
				\psi(x,y) := |\{ n \leq x : n \textup{ is } y-\textup{smooth}\}|.
			\]
		\end{defn}
		Smooth integers have the following well-known
                distribution~\cite[Thm. 10,
                p. 97]{TMF-PrimesAndDistributions}
%\felipe{Instead of $u^{-u}$ one has Dickman's
%function $\rho(u)$. Asymtotically they are the same. Will that
%matter?}
%\kurlberg{Yes.  Fixing it.}
for  $\alpha \in (0,1)$ and $x$ tending to infinity,
		\[
			\psi(x,x^{\alpha}) = (1+o(1)) \cdot x \rho(u)
		\]
		where \change{$u := \log x/\log x^\alpha = 1/\alpha$},  and $\rho(u) \in
                 (0,1)$ for $u \in (1,\infty)$.

		Our analysis will be based on the following heuristic:
                that $|C_{p}|$ has the same ``likelihood'' of being
                smooth as a random integer of the same size.
                By~\eqref{eq:A1}, $|C_{p}| = p^{d_{1}/2+o(1)}$, so the
                heuristic implies that the density of primes $p$ for
                which $|\cyclicp| = p^{d_{1}/2+o(1)}$ is
                $p^{\alpha}$-smooth equals $\rho(u)$ where $u = \log
              p^{d_{1}/2+o(1)}/\log p^{\alpha} = d_{1}(1+o(1))/(2\alpha)$, and hence that
		\[
			\prod_{\substack{p \leq x\\|C_{p}| \textup{ is }x^{\alpha} 
			\textup{-smooth}}} p = 
			\exp\left(\sum_{\substack{p \leq x\\
			|C_{p}| \textup{ is }x^{\alpha}\textup{-smooth}}} \log p\right)
			= \exp\left(x \cdot \rho
                          \left(\frac{d_1(1+o(1))}{2\alpha} \right) \right).
		\]
		This heuristic leads us to the following precise 
		\defi{cycle length smoothness assumption}: For
                $\alpha$ and $d_{1}$ fixed and $x\to \infty$, 
		\begin{equation} \label{eq:A4}
			\prod_{\substack{p \leq x\\|C_{p}| \textup{ is }
			x^{\alpha}\textup{-smooth}}} p = 
			\exp\left(x \cdot \rho
                          \left( \frac{d_1(1+o(1))}{2\alpha} \right) \right).
			% \left(x \cdot \frac{d_1(1+o(1))}{2\alpha} \right).
		\end{equation}

		\begin{proposition}\label{prop:ProbResult}
			Assume~\eqref{eq:A1}--\eqref{eq:A4}. Then there exists a sequence of  \emph{squarefree} integers $m$ such that
			\[
				\prob( \Cm \cap V(\Z/m\Z) = \emptyset  ) = 1 - o(1)
			\]
			as $m \to \infty$.
		\end{proposition}
		\begin{proof}
			Define
				\[
					m_{x,\alpha} := \prod_{\substack{p \leq x\\|\cyclicp| 
					\textup{ is }x^{\alpha}\textup{-smooth}}} p.
				\]
			Since for any square free integer $M$, $|C_M|
                        = \lcm_{p|M}  |\cyclicp|$, we have, as $x \to \infty$,
			\begin{align*}
				|C_{m_{x,\alpha}}| \leq & \prod_{p \leq x^{\alpha}} 
				p^{\log x^{d_{1}/2+o(1)}/\log(p)} =
				\exp\left(\sum_{p \leq x^{\alpha}} (d_{1}/2+o(1))\log x  \right)\\
				=& \exp\left( (d_{1}/2+o(1)) \cdot\log x \cdot \pi(x^\alpha) \right)
				= \exp\left( x^{\alpha} \cdot \frac{d_{1} \cdot (1+o(1))}{2\alpha} 
				\right).
			\end{align*}
			In particular, $|C_{m_{x,1/3}}| = \exp(O({x}^{1/3} )).$
			Therefore, if we take 
			\[
				m = m_{x,1/3} = \prod_{\substack{p \leq x\\
				|\cyclicp| \textup{ is }x^{1/3}\textup{-smooth}}} p
			\]
			% we find that $m \gg \exp(d_{1} x)$, thus 
			we find that $m = \exp(\Theta(x))$, thus 
			$|\Cm| = \exp(O({x}^{1/3})) = m^{o(1)}$, and hence
			\begin{align*}
				\prob( \Cm \cap V(\Z/m\Z) = \emptyset) = &
				\left(1 - \frac{1}{m^{d_{1}-d_{2}}} \right)^{|\Cm|}\\
				= & \left(1 - \frac{1}{m^{d_{1}-d_{2}}}
				\right)^{m^{d_{1}-d_{2}}\frac{|\Cm|}{m^{d_{1}-d_{2}}}} 
				= e^{-\frac{|C_m|}{m^{d_1 - d_2}}}
				= e^{-o(1)} = 1-o(1).
			\end{align*}
			as $x \to \infty$.
			\end{proof}

%%%%%%%%%%%%%%%%%%%%%%%%%%%%%%%%%%%%%%%%%%%%%%%%%%%%%%%%%%%%%%%%%%%%%%%%%%%%%%%%
\section{Computations}\label{sec:computations}%%%%%%%%%%%%%%%%%%%%%%%%%%%%%%%%%%
%%%%%%%%%%%%%%%%%%%%%%%%%%%%%%%%%%%%%%%%%%%%%%%%%%%%%%%%%%%%%%%%%%%%%%%%%%%%%%%%

	\subsection{Cycle length smoothness assumption}

	%In this section, we give some numerical evidence to justify the cycle length smoothness assumption~\eqref{eq:A4}.
	We ran experiments, detailed below, to justify the assumption~\eqref{eq:A4} on the smoothness of the cycle lengths. 
	Our experiments do not confirm this assumption. Fortunately, it is clear from the proof of proposition \ref{prop:ProbResult} 
	that we only need that the cycle lengths be at least as smooth as the prediction~\eqref{eq:A4} and this is 
	what we see in the experiments. We also found some maps with special properties for which the cycle lengths are even smoother. 
	We conjecture that cycle lengths be at least as smooth as the prediction~\eqref{eq:A4} in all cases but we don't know how 
	to explain the extra smoothness shown in the experiments.

	We will consider three maps:
	\begin{align*}
		\varphi\colon \PP^1\to\PP^1, & \quad (x:y) \mapsto (x^2 + 5y^2: y^2),\\
		\psi\colon\PP^2\to\PP^2, &\quad \left( x: y: z\right) \mapsto \left( x^2+y^2:x^2+3y^2-2xy+z^2: z^2 \right),\textup{ and}\\
		\sigma\colon\PP^3\to\PP^3, & \quad \left(x:y:z:w\right) \mapsto \left(x^2+y^2-z^2 + yw+w^2: x^2-xy+xz+2w^2:\right. \\
		& \hspace{7.5cm} \left. z^2-yz+xz+3w^2:w^2 \right).
	\end{align*}
For each prime less than 100,000, 500,000 and 1,000,000 respectively we compute $C_p$, the length of the periodic cycle length of 
$\left[1:1\right]$, $\left[1:1:1\right]$, or $\left[1:1:1:1\right]$ under $\varphi$, $\psi$, and $\sigma$ respectively in $\mathbb{F}_p$. 
Setting $\alpha = 1/3$, we compute $S \left( x \right) := \left[\displaystyle\prod_{\substack{p \leq x\\|C_{p}| \textup{ is } x^{\alpha}\textup{-smooth}}} p \right]$ 
at each prime in the range specified. We then create the graphs below 
which compare $\log S\left( x \right)$ to the predicted value of $x \cdot \rho \left( u \right)$ where $u = \frac{d_1}{2 \alpha}$ and $\rho$ is the Dickman $\rho$-function.  
Recall that assumption~\eqref{eq:A4} states that $\log S(x)$ should
behave linearly; the graphs (figures 1, 2, and 3 below) support this assumption. 
The data appears linear for large enough $x$, with slope at least as big as predicted.

All computations were performed using C and Sage 4.8~\cite{sage}. 

\begin{figure}[ht]\label{dim1}
	\includegraphics[scale=.5]{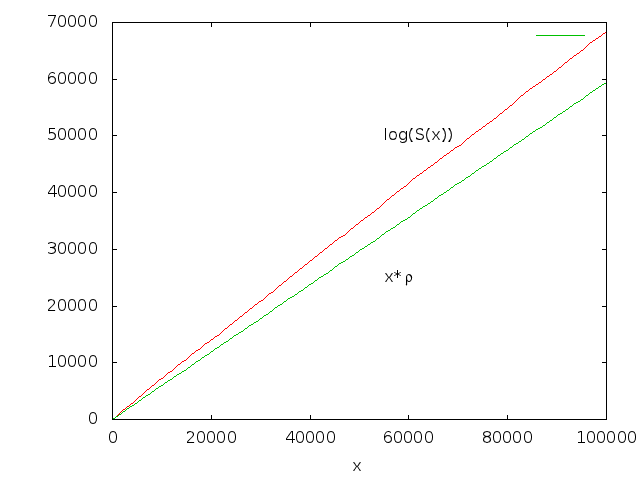}
	\caption{$\log \left( S \left(x \right) \right)$ for $\varphi \left( x \right) $ and $x<100,000$}
\end{figure}
 
\begin{figure}[h!]\label{dim2}
	\includegraphics[scale=.5]{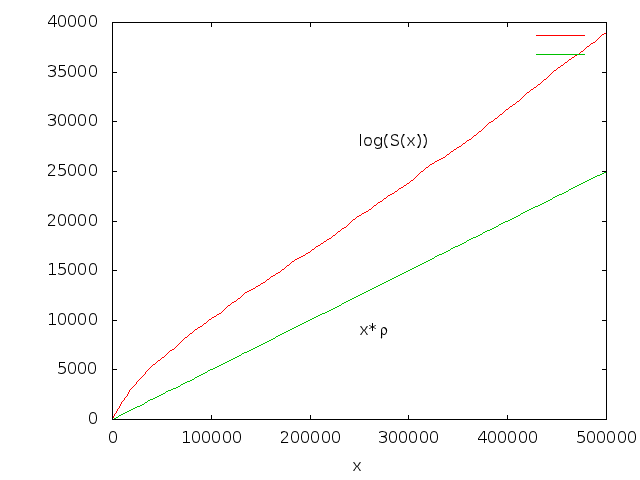}
	\caption{$\log \left( S \left(x \right) \right)$: for $\psi \left( x \right)$ and $x<500,000$}
\end{figure}
 
\begin{figure}\label{dim3}
	\includegraphics[scale=.5]{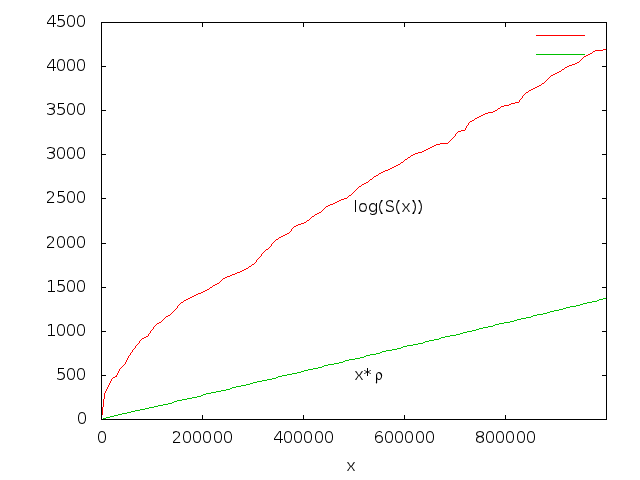}
	\caption{$\log \left( S \left(x \right) \right)$ for $\sigma \left( x \right)$ and $x<1,000,000$}
\end{figure}

\newpage
	%%%%%%%%%%%%%%%%%%%%%%%%%%%%%%%%%%%%%%%%%%%%%%%%%%%%%%%%%%%%%%%%%%%%%%%%%%%%
\subsection{Experiments}%%%%%%%%%%%%%%%%%%%%%%%%%%%%%%%%%%%%%%%%%%%%%%%%%%%%
	%%%%%%%%%%%%%%%%%%%%%%%%%%%%%%%%%%%%%%%%%%%%%%%%%%%%%%%%%%%%%%%%%%%%%%%%%%%%
		In this section, we let $X = \A^5$ and let $V = V(1 - x^2 - y^2 - z^2 - w^2 - v^2)$.  Fix a point $P\in X(\Z)$ and consider a morphism $\phi\colon X \to X$ with integer coefficients; then $\OO_{\phi}(P)\subseteq X(\Z)$.  As $V$ contains few integral points, namely only those points with exactly one coordinate equal to $\pm 1$ and the remaining coordinates $0$, one expects the intersection $\calO_{\phi}(P)\cap V$ to be empty. Thus, by the arguments in~\S\ref{sec:prob_proof}, we expect to find an integer $m$ such that $O_m(P)\cap V(\Z/m\Z) = \emptyset.$
		
		We considered a fixed integer starting point $P$ and $500$ randomly generated morphisms from $X\to X$ with integer coefficients.  For each of these morphisms, we computed whether the intersection $O_m(P)\cap V(\Z/m\Z)$ was empty for some positive integer $m\leq 2000$.  The results are as follows:
		\begin{enumerate}
			\item For $86.8\%$ of these maps, i.e., $434$ out of $500$, there exists a prime power $q\leq 2000$ such that $O_q(P)\cap V(\Z/q\Z) = \emptyset$.
			\item For $96.2\%$ of these maps, i.e., $481$ out of $500$, there exists a positive integer $m\leq 2000$ such that $O_m(P)\cap V(\Z/m\Z) = \emptyset$.
			\item For $11$ of the $19$ remaining maps, we conclude that $V\cap\OO_{\varphi}(P) = \emptyset$ by finding an integer $2000<m<11,500$ such that $O_m(P)\cap V(\Z/m\Z) = \emptyset$.  In each of these cases, the smallest such $m$ is supported at more than $1$ prime.
			\item For the remaining $8$ maps, $O_m(P)\cap V(\Z/m\Z) \neq \emptyset$ for all $m < 11,500$.  However, we are still able to conclude that $V\cap\OO_{\varphi}(P) = \emptyset$ by showing that the image of $V(\Z)$ modulo $7$ is disjoint from $O_7(P).$
		\end{enumerate}
		All computations were performed using \texttt{Magma}~\cite{Magma}.
	
%%%%%%%%%%%%%%%%%%%%%%%%%%%%%%%%%%%%%%%%%%%%%%%%%%%%%%%%%%%%%%%%%%%%%%%%%%%%%%%%
\section{A dynamical Hasse principle for \'etale morphisms}%%%%%%%%%%%%%%%%%%%%%
\label{sec:DynamicalHasse}%%%%%%%%%%%%%%%%%%%%%%%%%%%%%%%%%%%%%%%%%%%%%%%%%%%%%%
%%%%%%%%%%%%%%%%%%%%%%%%%%%%%%%%%%%%%%%%%%%%%%%%%%%%%%%%%%%%%%%%%%%%%%%%%%%%%%%%

	%%%%%%%%%%%%%%%%%%%%%%%%%%%%%%%%%%%%%%%%%%%%%%%%%%%%%%%%%%%%%%%%%%%%%%%%%%%%
	\subsection{Notation}%%%%%%%%%%%%%%%%%%%%%%%%%%%%%%%%%%%%%%%%%%%%%%%%%%%%%%%
	%%%%%%%%%%%%%%%%%%%%%%%%%%%%%%%%%%%%%%%%%%%%%%%%%%%%%%%%%%%%%%%%%%%%%%%%%%%%
		
		Let $K$ be a global field, and let $M_K$ denote its set of places.  For a finite set $S$ of places of $K$ containing all the archimedean places, we write $\OO_{K,S}$ to denote the ring of $S$-integers.  For all $v\in M_K$, we use $K_v$ to denote the $v$-adic completion.  If $v$ is nonarchimedean, we write $\OO_v$, $\mm_v$, and $k_v$ for the valuation ring, maximal ideal, and residue field of $v$, respectively.
		
		Let $X$ denote a $K$-variety, i.e., a reduced separated scheme of finite type over $K$, let $V\subseteq X$ denote a closed $K$-subvariety, and let $\varphi \colon X\to X$ denote a $K$-endomorphism.  For any $K$-variety $Y$, define
		\begin{equation}\label{define_XKS}
			Y(K,S) := \prod_{v\notin S} Y(K_v).
		\end{equation}
		We equip $Y(K_v)$ with the $v$-adic topology and $Y(K, S)$ with the product topology.  We view $Y(K)$ as a subset of $Y(K,S)$ via the diagonal embedding.  For every subset $T$ of $Y(K,S)$ or $Y(K_v)$, we write $\cC(T)$ or $\cC_v(T)$ for the closure of $T$ in the product topology or $v$-adic topology, respectively.  
		
		Since $Y$ is separated, $Y(K_v)$, and hence $Y(K,S)$, is Hausdorff. We note that as $Y$ is not assumed to be projective, $Y(K,S)$ need not agree with the adelic points of $Y$.  For basic terminologies and properties of scheme theory, we refer the readers to \cite{Hartshorne}, for properties of smooth and \'etale morphisms used throughout this section, we refer the readers to \cite{EGA4}.
		
	%%%%%%%%%%%%%%%%%%%%%%%%%%%%%%%%%%%%%%%%%%%%%%%%%%%%%%%%%%%%%%%%%%%%%%%%%%%%
	\subsection{The dynamical Hasse principle for \'etale maps and preperiodic 
	subvarieties}%%%%%
	%%%%%%%%%%%%%%%%%%%%%%%%%%%%%%%%%%%%%%%%%%%%%%%%%%%%%%%%%%%%%%%%%%%%%%%%%%%%

		For any point $P\in X(K)$, we have the following containments:
		\[
			V(K)\cap 
			\calO_{\varphi}(P)\subseteq V(K,S)\cap\cC(\calO_{\varphi}(P))
			\subseteq \prod_{v\notin S} V(K_v)\cap \cC_v(\calO_{\varphi}(P)).
		\]
		Recall the definition by Hsia-Silverman \cite[p.237--238]{hsiasilverman}
		that $(X,V,\varphi)$ is said to be dynamical Brauer-Manin $S$-unobstructed 
		if the leftmost containment is an equality for every $P\in X(K)$ 
		satisfying $\cO_{\varphi}(P)\cap V^{\text{pp}}=\emptyset$, where
		$V^{\text{pp}}$ is the union of all positive dimensional
		preperiodic subvarieties of $V$. In analogy, we define the 
		dynamical Hasse principle:
		
		\begin{definition}\label{def:Hasse}
		The triple $(X,V,\varphi)$ is said to satisfy the dynamical Hasse principle 
		(over $K$) if
		for every $P\in X(K)$ such that $\cO_{\varphi}(P)\cap V(K)=\emptyset$,
		there exists a place $v$ (depending on $P$) such that 
		$\cC_v(\cO_{\varphi}(P))\cap 
		V(K_v)=\emptyset$.  If there are infinitely many such places $v$, 
		we say that
		$(X,V,\varphi)$ satisfies the strong dynamical Hasse principle.		 
		\end{definition}

		When $V=V^{\text{pp}}$, if $(X,V,\varphi)$ satisfies the strong
		dynamical Hasse principle then it is immediate that 
		$(X,V,S)$ is Brauer-Manin $S$-unobstructed for every $S$.
		The reason is that for every $P\in X(K)$ such that
		$\cO_{\varphi}(P)\cap V(K)=\emptyset$, both containments
		above are equalities since all the three sets are the empty set.
		Our main results in this section are the following:
		
		\begin{theorem}\label{thm:DynamicalHasse}
			Assume that $X$ is quasi-projective, that $\varphi$ is \'etale, and that $V$ is $\varphi^m$-invariant, i.e., that $\varphi^{m}(V) \subseteq V$, for some $m\in\Z_{>0}$.  If $P\in X(K)$ is such that $V(K)\cap \calO_{\varphi}(P)= \emptyset$ then, for all but finitely many primes $v$,
			\[
				V(K_v)\cap \cC_v(\calO_{\varphi}(P)) = \emptyset
			\]
			Consequently, $(X, V, \varphi)$ satisfies the strong dynamical Hasse 
			principle.
		\end{theorem}
		
		We obtain a similar result when $V$ is $\varphi$-preperiodic, 
under the mild additional
		assumption that $\varphi$ is closed:
		\begin{theorem}\label{thm:new_global}
			Assume that $X$ is quasi-projective and that $\varphi$ is \'etale and closed. 
			Let $V$ be a preperiodic subvariety of $X$, which means that  
there exist 
			$k\geq 0$ and $m>0$
			such that $\varphi^{k+m}(V)\subset \varphi^k(V)$.
			For every $P\in X(K)$, if $V(K)\cap \calO_{\varphi}(P)= \emptyset$ then, for all but finitely many primes $v$,
			\[
				V(K_v)\cap \cC_v(\calO_{\varphi}(P)) = \emptyset
			\]
			Consequently, $(X, V, \varphi)$ satisfies the strong dynamical Hasse 
			principle.
		\end{theorem}

Notice that in the theorems, we do not suppose that $V$ is irreducible. But 
it is clear that the general case reduces to this. Indeed, for instance in
\ref{thm:new_global}, let $C$ be an irreducible component of $V$ such that
the closure of the orbit of $P$ meets $C$. Though the component $C$
does not have to be preperiodic in general, $C$ is preperiodic when it is
of maximal dimension (thanks to the \'etaleness of $\varphi$). Moreover, if 
$C$ is not preperiodic, then some power of $\varphi$ sends it to a component
of $V$ of greater dimension. Thus, if the closure of the orbit of $P$ meets 
$C$, it necessarily meets some preperiodic component of $V$, so we may
replace $V$ by this component.

\medskip

		In~\S\ref{subsec:local} we prove a local version of Theorem~\ref{thm:DynamicalHasse}.  Next in~\S\ref{subsec:global}, we show how Theorem~\ref{thm:DynamicalHasse} follows from the local version, Theorem~\ref{thm:adic_closure}.  In~\S\ref{subsec:preinvariant} we prove Theorem~\ref{thm:new_global} and in~\S\ref{subsec:closing} we give some closing remarks.

	%%%%%%%%%%%%%%%%%%%%%%%%%%%%%%%%%%%%%%%%%%%%%%%%%%%%%%%%%%%%%%%%%%%%%%%%%%%%
	\subsection{The local statement}\label{subsec:local}%%%%%%%%%%%%%%%%%%%%%%%%
	%%%%%%%%%%%%%%%%%%%%%%%%%%%%%%%%%%%%%%%%%%%%%%%%%%%%%%%%%%%%%%%%%%%%%%%%%%%%
		Throughout this section, we work locally.  Let $A$ denote a complete discrete valuation ring, $\mm$ its maximal ideal, and $k$ its residue field; we will assume that $k$ is perfect.  We write $F$ for the fraction field of $A$.

		The goal of this section is to prove the following:
		\begin{theorem}\label{thm:adic_closure}
			Let $\scrX$ be a smooth quasi-projective scheme over $A$, let $\varphi$ be an \'etale endomorphism of $\scrX$, and let $\scrV$ be	 a reduced and closed subscheme of $\scrX$ such that $\varphi^{M}(\scrV)\subseteq\scrV$ for some $M\geq 1$.  Let $P\in \scrX(A)$ be such that the reduction $\bar{P}$ of $P$ modulo $\mm$ is preperiodic under $\varphi.$ If $\scrV(A)\cap\calO_{\varphi}(P)=\emptyset,$
			then $\scrV(A)$ does not intersect the $\mm$-adic closure of $\calO_{\varphi}(P)$.
		\end{theorem}
		
		We beging with a few preliminary lemmas.
		\begin{remark}
			Lemmas~\ref{lem:topology}--\ref{lem:F_i} below are similar to results in~\cite[Sec. 2]{BGT_AJM} which were proved for $A = \Z_p$.
		\end{remark}

		\begin{lemma}\label{lem:topology}
			Let $\scrX$ be a separated $A$-scheme of finite type. Then we have the following.
			\begin{itemize}
				\item [(a)] $\scrX(A)$ is an open and closed subset of $\scrX(F)$, and it is compact if $k$ is finite.
				
				\item [(b)] For each $x\in\scrX(A)$, and each positive integer $n$, 
				\[ 
					C(x,n) := \{x'\in\scrX(A) : 
					x' = x \textup{ in }\scrX(A/\mm^n)\} 			
				\]
				is open and closed and the collection $\{C(x,n):\ n>0\}$ forms a basis of neighborhoods of $x$. Furthermore, for any $n\in\Z_{>0}$, the set $\scrX(A)$ is a disjoint union of subsets of the form $C(x,n)$; if $k$ is finite, then $\scrX(A)$ is a disjoint union of \emph{finitely many} such subsets.
			\end{itemize}
		\end{lemma}
			
		\begin{proof}
			Note that the $\mm$-adic topology on $\scrX(F)$ is defined as follows. Let $\{U_{\alpha}\}$ be an affine open cover of $\scrX$; assume that each $U_{\alpha}$ is of the form $\Spec(A[T_1,\ldots,T_m]/I_{\alpha})$.  Then $U_\alpha(F)$ is identified with the zero locus of $I_\alpha$ in $F^m$, and we equip it with the usual $\mm$-adic topology. This topology is independent of the choice of an affine covering and independent of the choice of presentation $\Spec(A[T_1,\ldots,T_m]/I_{\alpha})$ of an open in the covering. Since, by assumption, $\scrX$ is of finite type, we assume that the covering is finite. All assertions in the lemma follow immediately from the above description of the $\mm$-adic topology on $\scrX(F)$.
		\end{proof}
			
		Now we will prove a series of lemmas that give an explicit presentation of the local ring at points of $\scrX$ and $\scrV$ in terms of power series rings with coefficients in $A$.

		\begin{lemma}\label{lem:complete_local}
			Let $\scrX$ be a smooth quasi-projective scheme over $A$, let $P$ be an $A$-point of $\scrX$, and let $\bar{P}$ denote its reduction in $\scrX(k)$. Then the $\mm_{\bar{P}}$-adic completion $\hat{\calO}_{\scrX, \bar{P}}$ is isomorphic to $A[[T_1, \ldots, T_g]]$, where $g$ is the relative dimension of $\scrX$ at $P$ over $A$.  Moreover, a choice of an isomorphism 
			\[
				\tau\colon \hat{\calO}_{\scrX, \bar{P}}\stackrel{\sim}{\to}
				A[[T_1, \ldots, T_g]]
			\]
			gives the corresponding homeomorphism
			\begin{equation}\label{eq:homeo}
				h_{\tau}:\ \changeK{C(P,1)}\cong \mm^g.
			\end{equation}
		\end{lemma}		
		\begin{proof}
			If $F$ is characteristic $0$ and $k$ is characteristic $p$, then the ring $\hat{\calO}_{\scrX,\bar{P}}$ is finite (as a module) over  $\mathcal{W}(k)[[T_1,\ldots,T_g]]$ where $\mathcal{W}(k)$ is the ring of Witt vectors over $k$~\cite[Ch.\ 10]{Mat86}. Since the point $P$ gives rise to a section of the structural morphism:
			\[
				A\longrightarrow\hat{\calO}_{\scrX,\bar{P}}
			\]
			we have $\hat{\calO}_{\scrX,\bar{P}}$ is isomorphic to $A[[T_1,\ldots,T_g]]$.

			If $F$ and $k$ have equal characteristic, we have $\hat{\calO}_{\scrX,\bar{P}}$ is isomorphic to $k[[T_1,\ldots,T_{g+1}]]$, and $A$ is isomorphic to $k[[T]]$. Again, the point $P$	makes sure that the image of $T$ in $\hat{\calO}_{\scrX,\bar{P}}$ lies in $\hat{\mm}-\hat{\fm}^{2}$ where $\hat{\fm}$ is the maximal ideal of $\hat{\calO}_{\scrX,\bar{P}}$. Thus the structural morphism makes $\hat{\calO}_{\scrX,\bar{P}}$ a power series in $g$ variables with coefficients	in $A$.

			Every point	$u$ in $C(P,1)$ is equivalent to an $A$-morphism:
			\[
				\tilde{u}:\ \Spec A\rightarrow
				\Spec\calO_{\scrX,\bar{P}},
			\]
			which is in turn equivalent to specifying an $A$-algebra	homomorphism:
			\[
				\tilde{u}^{\sharp}:\ \hat{\calO}_{\scrX,\bar{P}} \rightarrow A.
			\]
			Thus, for a fixed isomorphism $\tau\colon \hat{\calO}_{\scrX, \bar{P}}\stackrel{\sim}{\to}A[[T_1, \ldots, T_g]]$, we have that giving $\tilde{u}^{\sharp}$ is equivalent to giving an element:
			\[
				h_{\tau}(u) := 
				(\tilde{u}^{\sharp}(T_i))_{i = 1}^g\in\mm^g.
			\]
		\end{proof}

		\begin{lemma}\label{lem:reduced_completion}
			Retain the notation from Lemma~\ref{lem:complete_local}. Assume that $\scrV$ is a closed subscheme of $\scrX$ such that $\bar{P}\in \scrV(k)$; write $\scrI$ for the ideal sheaf of $\scrV$.  After fixing an isomorphism $\tau\colon \hat{\calO}_{\scrX, \bar{P}}\stackrel{\sim}{\to}A[[T_1, \ldots, T_g]]$ (which exists by Lemma~\ref{lem:complete_local}), we may view $\scrI_{\bar{P}}\hat{\calO}_{\scrX,\bar{P}}$ as an ideal of $A[[T_1,\ldots,T_g]]$.  Then the homeomorphism $h_\tau$ in~\eqref{eq:homeo} maps $\scrV(A)\cap C(P,1)$ onto 
			\[
				\{\beta\in\mm^{g}:\ f(\beta)=0\ \ 
				\forall f\in\scrI_{\bar{P}}
				\hat{\calO}_{\scrX,\bar{P}}\}.
			\] 
			Furthermore, the $\fm_{\bar{P}}$-adic completion 
			\[
				\hat{\calO}_{\scrV,\bar{P}}\cong 
				\hat{\calO}_{\scrX,\bar{P}}/
				\scrI_{\bar{P}}\hat{\calO}_{\scrX,\bar{P}}
			\]
			is reduced if $\scrV$ is reduced at $\bar{P}$.
		\end{lemma}
		\begin{proof}
			The first assertion is immediate. The fact that 
			$\hat{\calO}_{\scrV,\bar{P}}\cong 
				\hat{\calO}_{\scrX,\bar{P}}/
				\scrI_{\bar{P}}\hat{\calO}_{\scrX,\bar{P}}$
			is a well-known property: taking completion preserves exact sequences of finitely generated modules over a Noetherian ring (see, for example~\cite[p.\ 108]{AM}). Since the local ring $\calO_{\scrV,\bar{P}}$ is reduced and excellent, its completion is reduced \cite[Ch.\ 11]{Mat86}.
		\end{proof}

		\begin{lemma}\label{lem:F_i}
			Retain the notation from Lemma~\ref{lem:complete_local} and fix an isomorphism $\tau\colon \hat{\calO}_{\scrX, \bar{P}}\stackrel{\sim}{\to}A[[T_1, \ldots, T_g]]$. Assume that $\varphi$ is an endomorphism of $\scrX$ such that $\overline{\varphi(P)}=\bar{P}$.  Then there are power series $f_1,\ldots,f_g$ in $A[[T_1,\ldots,T_g]]$ such that $f=(f_1,\ldots,f_g)$ fits into the following commutative diagram:
	       	\begin{equation}\label{diagram2}
	       		\xymatrix{
	       	       	C(P,1) \ar[r]^{\varphi} \ar[d]^{h_{\tau}} &
					C(P,1)\ar[d]^{h_{\tau}}\\
					\mm^g\ar[r]^{f} &
					\mm^g
				}
			\end{equation}
		\end{lemma}

		\begin{proof}
			Using $\tau$ to identify $\hat{\calO}_{\scrX,\bar{P}}$ with $A[[T_1,\ldots,T_g]]$, we define:
			\[
				f_i=\widehat{\varphi^{*}}T_i\ \ \forall 1\leq i\leq g,
			\]
			where $\widehat{\varphi^{*}}:\ \hat{\calO}_{\scrX,\bar{P}}\rightarrow\hat{\calO}_{\scrX,\bar{P}}$.  Then the commutativity is immediate.
		\end{proof}

		We will also use the following easy lemma:
		\begin{lemma}\label{lem:map_I_to_I}
			Let $R=A[[T_1,\ldots,T_g]]$, and $\varphi\colon R\rightarrow R$ a ring isomorphism. If $I$ is a radical ideal of $R$ such that $\varphi(I)\subseteq I$, then $\varphi(I)=I$.
		\end{lemma} 
		\begin{proof}
			Although the lemma is stated for a particular ring $R$, it will be clear from the proof that we require only that $R$ is a Noetherian equidimensional catenary ring (see~\cite[pp. 452, 453]{Eisenbud-CA} for the definition). Let $P_1,\ldots,P_n$ be the minimal primes containing $I$; then $\displaystyle I=\bigcap_{j=1}^{n}P_j$. Let $\displaystyle D=\max_{1\leq j\leq n}\dim R/P_j$. We will prove that for any integer $d$, $\varphi$ fixes the set 
			\[
				\calS_d:=\{P_j:\ \dim R/P_j=d\};
			\]
			the proof uses backwards induction on $0\leq d\leq D$. From this statement, the lemma follows immediately.

			Since, for each $1 \leq j\leq n$
			\[
				\bigcap_{i=1}^{n} \varphi(P_i)=\varphi(I)
				\subseteq I\subseteq P_j,
			\]
			there exists some $i$ (depending on $j$) such that $\varphi(P_i)\subseteq P_j$.  Let $P$ be an arbitrary prime in $\calS_D$. Then there is some $P_i$ such that	$\varphi(P_i)\subseteq P$. We have:
			\[
				D = \dim R/P \leq \dim R/\varphi(P_i) = \dim R/P_i.
			\]
			Since $D$ is the maximal dimension, we must have equality and so $\varphi(P_i)=P$. Then a counting argument shows that $\varphi$ fixes $S_D$.

			Now fix $1\leq D'< D$ and assume that $\varphi$ fixes $S_d$ for all $D'< d\leq D$. Let $P$ be an arbitrary prime in $S_{D'}$, and let $1\leq i\leq n$ be such that $\varphi(P_i)\subseteq P$. If $d:=\dim R/P_i$ is strictly greater that $D'$, then, by the induction hypothesis, $\varphi(P_i) = P_j$ for some $j$. Since $P_j$ and $P$ are two distinct minimal primes containing $I$ and $P_j \subseteq P$, this yields a contradiction. Thus, we have
			\[
				D' = \dim R/P \leq \dim R/\varphi(P_i)=
				\dim R/P_i = d \leq D',
			\]
			and so $\varphi(P_i) = P$.  Again, a counting argument shows that $\varphi$ fixes $S_{D'}$, as desired. 
		\end{proof}

		Now we are equipped to prove Theorem~\ref{thm:adic_closure}.				
		\begin{proof}[Proof of Theorem~\ref{thm:adic_closure}]
			Assume that $\scrV\cap \calO_{\varphi}(P) = \emptyset$.  We first perform a few reductions.  By assumption, $\varphi^{M}(\scrV)\subseteq \scrV$ for some positive integer $M$; after replacing the data $(\varphi,P)$ with $(\varphi^{M},\varphi^{i}(P))$ for each $0\leq i<M$, we may assume that $M = 1$. We have also assumed that $P$ is $\varphi$-preperiodic modulo $\mm$, i.e., that $\varphi^{N_0 + N}(P) \equiv \varphi^{N_0}(P) \pmod{\mm}$.  Since $\scrV\cap \calO_{\varphi}(P) = \emptyset$, the $\mm$-adic closure of $\calO_{\varphi}(P)$ intersects $\scrV$ if and only if $\mm$-adic closure of $\calO_{\varphi}(\varphi^{N_0}(P))$ intersects $\scrV$.  Therefore, by replacing $P$ with $\varphi^{N_0}(P)$, we may assume that $\bar{P}$ is $\varphi$-periodic of period $N$. For each $0\leq i\leq N-1$, by replacing the data $(\varphi,P)$ with the data $(\varphi^N,\varphi^{i}(P))$, we may assume that $\varphi$ fixes $\bar{P}$. By Lemma~\ref{lem:complete_local}, we may identify $\hat{\calO}_{\scrX,\bar{P}}$ and $C(P, 1)$ with $A[[T_1,\ldots,T_g]]$ and $\mm^g$, respectively. By Lemma~\ref{lem:F_i}, there exist power series $f=(f_1,\ldots,f_g)$ such that $\varphi$ acts by $f$ in the chart $\mm^g$. 

			If $\bar{P}\notin\scrV(k)$, then $C(P):=C(P,1)$ contains the $\mm$-adic closure of $\calO_{\varphi}(P)$ and is disjoint from $\scrV(A)$, so	we are done. Now suppose that $\bar{P}\in\scrV(k)$. As in Lemma~\ref{lem:reduced_completion}, let $\scrI$ denote the ideal sheaf of $\scrV$. Let $H_1,\ldots,H_m$ denote a choice of generators of the ideal $\scrI_{\bar{P}}\hat{\calO}_{\scrX,\bar{P}}$ of $A[[T_1,\ldots,T_g]]$. Because $\varphi(\scrV)\subseteq\scrV$, we have $\varphi^{*}(\scrI_{\bar{P}})\subseteq\scrI_{\bar{P}}$ where	$\varphi^{*}:\ \calO_{\scrX,\bar{P}}\rightarrow\calO_{\scrX,\bar{P}}$ is the induced map on the stalk at $\bar{P}$.	Therefore $\widehat{\varphi^{*}}(\scrI_{\bar{P}})\hat{\calO}_{\scrX,\bar{P}}	\subseteq\scrI_{\bar{P}}\hat{\calO}_{\scrX,\bar{P}}$. By Lemma~\ref{lem:reduced_completion}, $\scrI_{\bar{P}}\hat{\calO}_{\scrX,\bar{P}}$ is a radical ideal of $\hat{\calO}_{\scrX,\bar{P}}$, and, since $\varphi$ is \'etale, $\widehat{\varphi^{*}}:\ \hat{\calO}_{\scrX,\bar{P}}\rightarrow \hat{\calO}_{\scrX,\bar{P}}$ is a ring isomorphism.  Thus, we may apply Lemma~\ref{lem:map_I_to_I} to conclude that $\widehat{\varphi^{*}}(\scrI_{\bar{P}})\hat{\calO}_{\scrX,\bar{P}} = \scrI_{\bar{P}}\hat{\calO}_{\scrX,\bar{P}}$.
			Therefore, there exist power series $c_{ij}\in A[[T_1,\ldots,T_g]]$, where $1\leq i,j\leq m$, such that
			\begin{equation}\label{eq:main_eq}
				H_i = \sum_{j=1}^{m}c_{ij}H_j(f_1,\ldots,f_g)
				\quad \forall 1\leq i\leq m.
			\end{equation}

			For a vector $v=(v_i)$ in $(F)^m$, define:
			\[
				||v||=\max_{1\leq i\leq m}|v_i|,
			\]
			and	for $u\in \mm^g$, write $H(u) = (H_i(u))\in A^m$. By~\eqref{eq:main_eq}, we now have:
			\begin{equation}\label{eq:main_ineq}
				 ||H(f(u))||\geq ||H(u)|| \quad \forall u=(u_i)\in \mm^g.
			\end{equation}
			Thus, the closure of $\calO_{\varphi}({P})$ is contained in 
			\[
				\{u\in\mm^g:\ ||H(u)||\geq ||H(P)||\}.
			\]
			Since $||H(P)|| > 0$ and $H$ vanishes on $\scrV(A)$, this completes the proof.
		\end{proof}

		\begin{rema}
	        If $\varphi(\scrV)\subseteq\scrV$, then we can prove a stronger statement than~\eqref{eq:main_ineq}, namely:
			\begin{equation}\label{eq:temp_eq}
			||H(F(u))||=||H(u)||. 
			\end{equation}
			The proof proceeds in the same way as the proof of~\eqref{eq:main_ineq}. More explicitly, we may write each $H_i(f_1,\ldots,f_g)$ as a linear combination of the $H_j$'s and obtain $||H(f(u))||\leq ||H(u)||$. 
			
			Morally speaking, if we consider $||H(P)||$ as the ``distance'' from $P$ to $\scrV$ then equality~\eqref{eq:temp_eq} means that when $P$ comes close to $\scrV$ (i.e. $\bar{P}\in\scrV(k)$) and $\varphi$ fixes $P$ modulo $\mm$, then all elements in $\OO_{\varphi}(P)$ are equidistant to $\scrV$. Thus $P$ becomes a ``satellite'' of $\scrV$ by moving around $\scrV$ under $\varphi$.
		\end{rema}
		
		If $F$ is a finite extension of $\Q_p$ for some prime $p$, then there is another proof of Theorem~\ref{thm:adic_closure} using a uniformization of $\cO_{\varphi}(P)$ defined as follows.

		\begin{definition}\label{unif}
			Retain the notation from Theorem~\ref{thm:adic_closure}.  Assume that $F$ is a finite extension of $\Q_p$ and that $\varphi$ fixes $P$ modulo $\mm$. Using Lemma~\ref{lem:complete_local}, identify $\hat{\calO}_{\scrX,\bar{P}}$ with $A[[T_1,\ldots,T_g]]$, and $\changeK{C(P):= C(P,1)}$ with $\mm^g$.  The orbit $\calO_{\varphi}(P)$ has a \defi{uniformization} if there exist power series $G_1,\ldots,G_g$ in $F[[T]]$, convergent on $\Z_p$
such that:
			\begin{itemize}
				\item[(i)] $ (G_1(0),\ldots,G_g(0))= P$, which equals $0$ in $\mm^g$, and
				\item[(ii)]
				 $\varphi(G_1(z),\ldots,G_g(z))= (G_1(z+1),\ldots,G_g(z+1))\ \forall z\in \Z_p$.
			\end{itemize}
		\end{definition}
		\begin{proof}[Proof of Theorem~\ref{thm:adic_closure} using {uniformization}]
			This proof is similar to arguments developed in~\cite{BGT_AJM}.	
		
			As above, we reduce to the case that $M = 1$ (i.e. $\scrV$ is invariant) and assume that $\bar{P}\in\scrV(k)$. By \cite[Theorem 3.3]{BGT_AJM} or
			\cite[Theorem 7]{A}, there is a uniformization of $\scrO_{\varphi}(P)$ (possibly after replacing
			$\varphi$ by an iterate).
			  Let $G = (G_1, \ldots, G_g)$ be such a uniformization.  By Definition \ref{unif}, the $\mm$-adic closure of $\calO_{\varphi}(P)$ is contained in $G(\Z_p)$. If there is some $u\in \Z_p$ such that $G(u)\in\scrV(A)$, then $G(u+n)\in\scrV(A)$ for every natural number $n$.  Let $H=0$ be any of
			the equations defining $\scrV$ in $\scrX$. Then we have
			$H\circ G(u+n)=0$ for 
		        every natural number $n$.
			Since a nonzero $p$-adic analytic function on $\Z_p$ can have only finitely many zeros, the analytic function $H\circ G$ must be identically zero on $\Z_p$. Therefore $G(\Z_p)\subseteq\scrV(A)$
and so the whole orbit of $P$ is contained in $\scrV(A)$, contradicting our assumption that
$\cO_{\varphi}(P)\cap \scrV(A)=\emptyset$.  
		\end{proof}

		While the argument using uniformization is simpler and might be applicable to some cases where $\varphi$ is not \'etale, our approach still has a number of advantages.
			\begin{itemize}
				 \item The proof of Theorem~\ref{thm:adic_closure} that does not require on uniformization can be translated into a simple and effective algorithm, hence	is suitable for computational purposes. 

				 \item Lemma~\ref{lem:map_I_to_I} and results of the same
				 type may be of independent interest. Indeed, results of a similar spirit are studied in a recent preprint of Bell and Lagarias~\cite{BL}.
				
				 \item Our proof requires no assumption on the characteristic of $F$, whereas uniformization of orbits does not hold if $\Char(F)>0$ (for example, it is impossible to have an ``exponential function'' since dividing by $n!$ is prohibited).
			\end{itemize}

	%%%%%%%%%%%%%%%%%%%%%%%%%%%%%%%%%%%%%%%%%%%%%%%%%%%%%%%%%%%%%%%%%%%%%%%%%%%%
	\subsection{Proof of Theorem~\ref{thm:DynamicalHasse} }%%%%%%
	\label{subsec:global}%%%%%%%%%%%%%%%%%%%%%%%%%%%%%%%%%%%%%%%%%%%%%%%%%%%%%%%
	%%%%%%%%%%%%%%%%%%%%%%%%%%%%%%%%%%%%%%%%%%%%%%%%%%%%%%%%%%%%%%%%%%%%%%%%%%%%
		In this section, we present the proof of Theorem~\ref{thm:DynamicalHasse}.  First we show that for all but finitely many places $v$, the assumptions of Theorem~\ref{thm:adic_closure} hold.
		\begin{lemma}\label{lem:has_model}
			Assume that $\varphi$ is \'etale and that $\varphi(V)\subseteq V$.  Fix a point $P\in X(K)$, then there exists a finite set $S\subseteq M_K$ containing all the archimedean places such that $X, V,\varphi$ and $P$ extend to models $\scrX, \scrV, \tilde\varphi$ and $\mathscr{P}$, respectively, over $\OO_{K, S}$ with the following properties:
			\begin{itemize}
				\item $\tilde\varphi$ is \'etale, 
				\item $\scrX$ is quasi-projective and smooth, and
				\item $\scrV$ is a reduced, closed subscheme of $\scrX$, is flat over $\OO_{K,S}$, and $\tilde\varphi(\scrV)\subseteq \scrV.$
			\end{itemize}
		\end{lemma} 
		\begin{proof}
			We first find models for $X$, $V$, $\varphi$ and $P$ over, say, some $\calO_{K,S}$. As the locus in $\Spec(\calO_{K,S})$ over which $\scrX$ is not smooth, $\scrV$ is not flat, or $\tilde{\varphi}$ is not \'etale is closed, by enlarging $S$, we may assume that this locus is empty. Since	$\scrV$ is flat over $\calO_{K,S}$ and its generic fiber is reduced, $\scrV$ is itself reduced.
			 
			It remains to ensure that $\tilde{\varphi}(\scrV)\subseteq\scrV$.  By enlarging $S$ again, we may assume that every irreducible component of $\scrV$ contains some point in the generic fiber. Then since $\varphi(V)\subseteq V$ and $V$ is dense in $\scrV$, we have $\tilde{\varphi}(\scrV)\subseteq \scrV$.   
		\end{proof}
		\begin{remark}
			The proof of Theorem~\ref{thm:DynamicalHasse} is simpler in the case when $X$ is smooth.  For clarity, we present the proof under this additional assumption first, and then give the general case.
		\end{remark}
		\begin{proof}[Proof of Theorem~\ref{thm:DynamicalHasse} under the assumption that $X$ is smooth]
			Fix $S\subseteq M_K$ as in Lemma~\ref{lem:has_model}. For $v\notin S$, let $\calO_{(v)} = K\cap\calO_v$ be the valuation ring of $v$ in $K$. By Lemma~\ref{lem:has_model}, there exist models $\scrX$ and $\scrV$ of $X$ and $V$, respectively, over $\calO_{(v)}$.  We take the formally smooth base change
			from $\calO_{(v)}$ to $\calO_v$, and abuse notation by (still) using $\scrV$ and $\scrX$
			to denote the resulting models over $\calO_{v}$.
						
			Now we apply Theorem~\ref{thm:adic_closure} to show that $\scrV(\calO_v)$ does not intersect the closure of $\calO_{\varphi}(P)$ in $\scrX(\calO_v)$.  Since $\scrX(\calO_v)$ is closed in $\scrX(K_v) = X(K_v)$, $V(K_v)=\scrV(K_v)$ does not intersect the closure of $\calO_{\varphi}(P)$ in $X(K_v)$. This completes the proof in the case when $X$ is smooth.
		\end{proof}

		\begin{proof}[Proof of Theorem~\ref{thm:DynamicalHasse} in the general case]
			Let $M>0$ be such that $\varphi^M(V)\subseteq V$. As in the proof of Theorem~\ref{thm:adic_closure}, by replacing the data $(\varphi,P)$ with the data 	$(\varphi^{M},\varphi^i(P))$ for $0\leq i<M-1$, we may assume that $M = 1$. We prove the theorem by Noetherian induction. 
			
			If $\dim(X) = 0$, there is nothing to prove. Assume that $\dim(X) > 0$; let $X'$ be the smooth locus of $X$, and let $X''= X - X'$ with the reduced closed subscheme structure. Since $\varphi$ is \'etale, $\varphi$ preserves both $X'$ and $X''$. If $P\in X''(K)$ then we use the induction hypothesis on $X''$. 
			
			Now suppose that $P\in X'(K)$. Write $V'=V\cap X'$. By Lemma~\ref{lem:has_model}, there exists a finite set of places $S$ such that $X$ and $V$, respectively, have flat (hence reduced) models $\scrX$ and $\scrV$ over $\calO_{K,S}$.  In addition, $\varphi$ and $P$ both extend to $\OO_{K,S}$ (by an abuse of notation, we denote these extensions by $\varphi$ and $P$, respectively), and $\varphi(\scrV)\subseteq\scrV$.  Since $\varphi$ is \'etale at $P$ generically, by enlarging $S$, we may assume that $\varphi$ is \'etale at every point in $P$. 
			
			Let $\scrX'$ denote the open subscheme of $\scrX$ consisting of smooth points over $\calO_{K,S}$, and let $\scrV'=\scrV\cap\scrX'$ be a model of $V'$ over $\calO_{K,S}$. Now let $v$ be a place outside $S$ and pull back all models to $\calO_v$. By Theorem~\ref{thm:adic_closure}, $\scrV'(\calO_v)$ does not intersect the $v$-adic closure of $\calO_{\varphi}(P)$ in $\scrX'(\calO_v)$. Since, by Lemma~\ref{lem:topology}, $\changeK{\scrX'(\calO_v)}$ is quasi-compact and $X(K_v)$ is Hausdorff, the set $\scrX'(\calO_v)$ is closed in $X(K_v)$. Therefore $V(K_v)$ does not intersect the $v$-adic closure of $\calO_{\varphi}(P)$ in $X(K_v)$.  This completes the proof.
		\end{proof}

	%%%%%%%%%%%%%%%%%%%%%%%%%%%%%%%%%%%%%%%%%%%%%%%%%%%%%%%%%%%%%%%%%%%%%%%%%%%%
	\subsection{Proof of Theorem~\ref{thm:new_global}}%%%%%%%%%%%%%%%%%%%%%%%%%%
	\label{subsec:preinvariant}%%%%%%%%%%%%%%%%%%%%%%%%%%%%%%%%%%%%%%%%%%%%%%%%%
		We proceed by induction on the dimension of $V$. 
		The case $\dim(V)=0$ is easy, as 
		follows. Let $U$ be the union of all iterates of $V$. If $\cO_{\varphi}(P)$
		intersects $U$ then $P$ is preperiodic and there is nothing to prove.
		Otherwise, we apply Theorem \ref{thm:DynamicalHasse} for $U$
		instead of $V$. 
		To carry out the induction step, we need the following lemma which might 
		be of independent interest.
		\begin{lemma}\label{lem:YcapY_1}
		Let $X$ and $\varphi$ be as in Theorem \ref{thm:new_global}. 
		Let $Y$ be an irreducible $\varphi$-preperiodic subvariety of $X$.
		Let $Y_1$ be a periodic iterate of $Y$. Then every irreducible component
		of $Y\cap Y_1$ is preperiodic.
		\end{lemma}
		\begin{proof}
		Replacing $\varphi$ by an iterate, we reduce to the case
		$Y_1=\varphi(Y)=\varphi^2(Y)\neq Y$. Let $\nu_1\colon 
		\tilde{Y_1}\rightarrow Y_1$ and
		$\nu\colon \widetilde{\varphi^{-1}(Y_1)}\rightarrow \varphi^{-1}(Y_1)$
		denote the normalizations of $Y_1$ and $\varphi^{-1}(Y_1)$, respectively. 
		For every integer $n\geq 1$, define:
		$$Z_n:=\left\{x\in Y_1\colon\ |\nu_1^{-1}(x)|\geq n\right\},$$
		where $|\nu_1^{-1}(x)|$ is counted with multiplicity. By the semicontinuity
		theorem, $Z_n$ is closed in $Y_1$.
		
		Now let $W$ be an irreducible component of $Y\cap Y_1$ having dimensional 
		$d$ and generic point $w$.
		Let $s=|\nu^{-1}(\varphi(w))|$, the main observation is that 
		$\varphi(W)$ is an irreducible component of $Z_s$. Assume otherwise and let
		$C_1$ be an irreducible component of $Z_s$ strictly containing
		$\varphi(W)$. Let $c_1$ denote the generic point of $C_1$.
		Since $c_1\in Z_s$, we have $|\nu_1^{-1}(c_1)|\geq s$.
		On the other hand, by semicontinuity $|\nu_1^{-1}(c_1)|\leq 
		|\nu_1^{-1}(\varphi(w))|=s$. Hence we must have $|\nu_1^{-1}(c_1)|=s$.
		
		By faithful
		flatness, the map
		$$\Spec(\cO_{V_1},\varphi(w))\longrightarrow \Spec(\cO_{\varphi^{-1}(V_1)},w)$$
		induced by $\varphi$ is surjective \cite[p.68]{AM}. Therefore, there 
		exists $c$ in $\varphi^{-1}(V_1)$ such that $\varphi(c)=c_1$
		and the Zariski closure $C$ of $c$ contains $W$. From the last
		paragraph and the fact that taking normalization commutes
		with \'etale base changes, we have:
		\begin{equation}\label{eq:wsc}
		|\nu^{-1}(w)|=|\nu_1^{-1}(\varphi(w))|=s=|\nu_1^{-1}(c_1)|=|\nu^{-1}(c)|
		\end{equation}
		
		By comparing dimensions, we see that $V$ and $V_1$ are two irreducible components 
		of $\varphi^{-1}(V_1)$ containing $w$. Since $C_1$ strictly contains
		$\varphi(W)$, we have that $C$ strictly contains $W$. 
		Note that it is impossible for both $V$ and $V_1$ to contain $C$ since
		$W$ is an irreducible component of $V\cap V_1$. Therefore
		the set of irreducible components of $\varphi^{-1}(V_1)$ containing
		$c$ is strictly smaller than the set of irreducible components
		of $\varphi^{-1}(V)$ containing $w$. Together with the semicontinuity theorem, we have
		$|\nu^{-1}(c)|<|\nu^{-1}(w)|$ contradicting (\ref{eq:wsc}).
		
		Now let $\calZ$ denote the union of all $d$-dimensional
		irreducible components of all the $Z_n$'s for $n\geq 1$.
		This is a finite union since $Z_n=\emptyset$ for all sufficiently large
		$n$ thanks to finiteness of $\nu_1$. We have proved
		that $\varphi(W)$ is an irreducible component of $\calZ$. 
		For each $m\geq 1$, repeat the same arguments for $\varphi^m$ instead
		of $\varphi$, we have that $\varphi^m(W)$ is 
		an irreducible component of $\calZ$. This proves that $W$ is preperiodic.
		\end{proof}
		
		We now return to the proof of Theorem \ref{thm:new_global}. We may assume
		that $V$ is irreducible and not periodic. Replacing $\varphi$
		by an iterate, we may assume $V_1=\varphi(V)=\varphi^2(V)\neq V$.
		
		If $\cO_{\varphi}(P)\cap V_1(K)=\emptyset$ then by Theorem 
		\ref{thm:DynamicalHasse}, there is a finite set of places
		$S_1$ such that $\calC_v(\calO_{\varphi}(P))\cap V_1(K_v)=\emptyset$
		for $v\notin S_1$. This implies
		$\calC_v(\calO_{\varphi}(P))\cap V(K_v)=\emptyset$ for every
		$v\notin S_1$.
		
		Now assume that $\varphi^m(P)\in V_1(K)$ for some $m\geq 0$. 
	  For every $v$, we have 
	  \[
	  	\calC_v(\calO_{\varphi}(\varphi^m(P)))\subseteq 
	  	V_1(K_v)
		\]
	  and 
	  \[
	  	\calC_v(\calO_{\varphi}(P))=\calC_v(\calO_{\varphi}(\varphi^{m}(P)))
	  	\cup\left\{P,\ldots,\varphi^{m-1}(P)\right\}.
	  \]
	  If $V(K_v)\cap\calC_v(\calO_{\varphi}(P))\neq\emptyset$,
	  then the above identities together with
	  $V(K)\cap \calO_{\varphi}(P)=\emptyset$ imply:
	  \[
		  \calC_v(\calO_{\varphi}(P))\cap (V\cap V_1)(K_v)\neq\emptyset.
		  \]
	  This can only happen for finitely many $v$'s by
	  applying the induction hypothesis for $V\cap V_1$, whose irreducible components
	  are preperiodic by Lemma \ref{lem:YcapY_1}.

	%%%%%%%%%%%%%%%%%%%%%%%%%%%%%%%%%%%%%%%%%%%%%%%%%%%%%%%%%%%%%%%%%%%%%%%%%%%%
	\subsection{Closing remarks}\label{subsec:Remarks}\label{subsec:closing}%%%%
	%%%%%%%%%%%%%%%%%%%%%%%%%%%%%%%%%%%%%%%%%%%%%%%%%%%%%%%%%%%%%%%%%%%%%%%%%%%%
		It is natural to ask whether the assumption in Theorem~\ref{thm:DynamicalHasse} that $V$ is preperiodic is necessary.  We show that if $K$ is a number field and $V = Q$ is a single $K$-point, then $(X, Q, \varphi)$ fails the strong Hasse principle if and only if $Q$ is not periodic but has \defi{almost everywhere periodic reduction}, i.e. for all but finitely many primes $\fp$, the reduction of $Q$ modulo $\fp$ is periodic.

\begin{prop}\label{prop:periodic-reduction}
 Let $K$ be a number field, let
$\varphi: X\to X$ be an \'etale $K$-endomorphism of $X$ and let $Q\in X(K)$ .
\begin{itemize}
 \item [(a)] Let $\fp$ be a prime
 such that we have models over $\cO_{\fp}$. If
 the reduction of $Q$ modulo $\fp$ is periodic then the $\fp$-adic closure of the orbit of
$\varphi(Q)$ contains $Q$.
 Consequently, if $Q$ is not periodic but has almost everywhere periodic reduction then
$(X, Q, \varphi)$ does not satisfy the strong Hasse principle.

\item [(b)]  Conversely, if $Q$ is either periodic or does not have almost everywhere periodic reduction, then $(X,Q,\varphi)$
satisfies the strong dynamical Hasse principle.  
	\end{itemize}
\end{prop}

\begin{proof}

(a) The first assertion follows immediately from
the $\fp$-adic uniformization of the $\varphi^N$-orbit of $Q$ (for some
$N>>0$ depending on $\fp$) and the fact
that for every analytic function $G$ from 
$\cO_\fp$ to $\cO_{\fp}^g$, the point $G(0)$ lies
in the closure of $\{G(1),G(2),\ldots\}$.
Such uniformization exists by \cite[Theorem 3.3]{BGT_AJM}, or more precisely
by its generalization
in \cite[Theorem 7]{A}. For the second assertion, note that the
orbit of $P=\varphi(Q)$ does not contain $Q$
but the $\fp$-adic closure of this orbit contains $Q$ for almost all $\fp$.

	\smallskip
	
(b) The case that $Q$ is
		periodic follows from Theorem \ref{thm:DynamicalHasse}.
		Hence we assume that $Q$ is non-periodic. Let $P\in X(K)$
		such that $Q\notin\cO_{\varphi}(P)$. 
		For every $\fp$ such that
		$X$, $\varphi$, $P$ and $Q$ have
		models over $\cO_{\fp}$,
		if $Q\in \calC_{\fp}(\cO_{\varphi}(P))$
		then $Q$ and
		$\varphi^m(P)$
		have the same reduction modulo $\fp$
		for $m$ as large as we like. This implies
		that $Q$ is periodic modulo $\fp$. But there are
		infinitely many primes
		such that this conclusion does not hold thanks to our assumption on $Q$, hence
		$(X,Q,\varphi)$ satisfies the strong dynamical Hasse principle.
\end{proof}

Some results in the literature suggest that the examples of non-periodic 
points with almost everywhere periodic reduction must be very special, and so
the strong dynamical Hasse principle mostly holds when the endomorphism
is \'etale and the subvariety is a single point. For instance, by a result
of Pink~\cite[Corollary 4.3]{Pink}, such points cannot exist for the 
multiplication-by-$d$ map on an abelian variety. Furthermore, by~\cite[Corollary
1.2]{BGHKST}, such points also cannot exist for a self-map of $\mathbb P^1$ of
degree at least two (though such a map is not  \'etale, and thus 
\ref{prop:periodic-reduction} does not apply directly).

On the other hand, the translation-by-one map on $\mathbb P^1$ is an obvious example where non-periodic points become periodic modulo almost all primes, so that the strong dynamical Hasse principle fails; more
generally, an automorphism of $\mathbb P^n$ of infinite order, given by an 
integer-valued matrix, has the same property. As it was 
pointed out to us by Serge Cantat, one can use this to construct other,
though somewhat artificial, examples: take an automorphism $\varphi$ of a 
smooth variety 
$X$ with a fixed point $x$, inducing an infinite-order automorphism on the
tangent space at $x$. Then $\varphi$ lifts to the blow-up of $X$ at $x$, inducing 
an automorphism of the exceptional divisor, and points of that exceptional 
divisor are periodic modulo almost all primes.

	It seems reasonable to conjecture that non-periodic points with almost everywhere periodic reduction do not exist for polarized morphisms $\varphi$ (that is, morphisms such that $\varphi^*{\mathcal L}={\mathcal L}^{\otimes k}$ for some integer $k>1$ and some ample line bundle ${\mathcal L}$), so that the strong dynamical Hasse principle holds for number fields $K$, \'etale polarized morphisms $\varphi$ and $V\in X(K)$. Notice however that \'etale polarized 
endomorphisms are extremely rare, cf.~\cite{Fakh}. 

		For the sake of completeness, we note that for curves over number fields, the only counterexamples to the dynamical Brauer-Manin 
criterion are automorphisms $\varphi$ of a very special kind.
		\begin{prop}\label{prop:BM_curve}
			Let $X$ be a smooth geometrically integral projective curve of genus $g$ over a number field $K$, let $\varphi$ be a nonconstant self-map of $X$ over $K$, and $V$ a finite subset of 
$X(K)$. We make the following additional assumptions.
			\begin{itemize}
				\item [(a)] If $g = 0$, assume that $\varphi$ is not conjugate to 
$z\mapsto z+1$,
				\item [(b)] If $g = 1$, assume that $\varphi$ has a preperiodic point in $X(\bar{K})$. (If we regard $X$ as an elliptic curve, this condition
				is equivalent to the condition that $\varphi$ is not
				a translate by a non-torsion point.)
			\end{itemize}
			Then we have the following equality:
			\[
				V(K)\cap\calO_{\varphi}(P) = 
				V(K,S)\cap\cC(\calO_{\varphi}(P)).
			\]
		\end{prop}
		
		\begin{remark}
			If $X =\PP^1$, $V$ is a finite set of points and $\deg(\varphi) \geq 2$, then Silverman and Voloch have shown that $\OO_\varphi(P)\cap V(K) = \calC(\calO_{\varphi}(P))\cap V(K, S)$~\cite[Theorem 1]{silvermanvoloch}.  
		\end{remark}
	\begin{remark} 
	When $X$ is an abelian variety, $V$ is an arbitrary subvariety and $\varphi$
	is a $K$-endomorphism of $X$ such that $\Z[\varphi]$ is an integral domain, 
	Hsia and Silverman
	show the equality:
	$$\cO_{\varphi}(P)\cap V(K)=\cC(\cO_{\varphi}(P))\cap V(K,S)$$
	under certain strong conditions. We refer the readers to \cite[Theorem 11]{hsiasilverman}
	for more details. Our proof of Proposition \ref{prop:BM_curve} gives 
	an unconditional proof of their result when $X$ is an elliptic curve.
	\end{remark}

		\begin{proof}
The case $g\geq 2$ is trivial since all endomorphisms of curves of genus at least two are of finite order.
			If $g=0$, this follows from~\cite{silvermanvoloch} as mentioned above (see
Remark 9 from \cite{silvermanvoloch} if $\deg(\varphi)=1$). 

			Now consider the case when $g=1$.  If $P$ is $\varphi$-preperiodic then there is nothing to prove, so we assume that $P$ is wandering.  There exists a non-negative integer $N$ 
such that $\varphi^M(P)\not\in V(K)$ for all $M>N$. 	After replacing	$P$ by $\varphi^{N+1}(P)$, we may assume the $\varphi$-orbit of $P$ does	not intersect $V(K)$. It remains to show that $V(K,S)\cap\cC(\calO_{\varphi}(P))=\emptyset$. By assumption (b), there is some $M>0$ such that $\varphi^M$ has a fixed point. By replacing the data $(\varphi, P)$ with $(\varphi^{M},\varphi^{i}(P))$ for $0\leq i<M$, we may assume that $\varphi$ has a fixed point.

			Note that if we can prove
			\[
				V(L,S_L)\cap\cC(\calO_{\varphi}(P)) = \emptyset
			\]
			for $L$ a finite extension of $K$, and $S_L\subset M_L$ the set of places of $L$ lying above places in $S$, then this implies that $V(K,S)\cap\cC(\calO_{\varphi}(P))=\emptyset$.  Thus, we may assume that there is a fixed point $O\in X(K)$. Using $O$ as the point at infinity
			 we have the following Weierstrass equation for $X$: 
			\[
				y^2 = x^3 + Ax + B.
			\]
			
			Let $U(z)/V(z)$ be the Latt\`es map associated to $\varphi$.  Since $O$ is fixed by $\varphi$, we have $\deg(V) < \deg(U) = \deg(\varphi)$.  Let $S'$ be a finite set of places such that $U, V\in \changeK{\OO_{K,S'}[z]}$ and such that the leading coefficients of $U$ and $V$ are in $\changeK{\OO_{K,S'}^{\times}}$.  Then
			\begin{equation}\label{eq:ValProp}
				\forall v\notin S'\ \forall Q\in X(K),\quad 
				\textup{ if }v(x(Q))<0, \textup{ then }v(x(\varphi(Q)))<0.
			\end{equation}
			Assume that $V(K,S)\cap\cC(\calO_{\varphi}(P))\neq\emptyset$ and let $a=(a_{\fp})_{\fp\notin S}$ be an element of the intersection. Note that since $V$ is a finite set of points, $V(K_{\fp})=V(K)$ and so $a_{\pp}\in V(K)$ for every $\fp\notin S$. By Theorem \ref{thm:new_global}, for almost all $\fp$, the $\fp$-adic closure of $\calO_{\varphi}(P)$ does not contain $O$. Therefore $a_{\fp}\neq O$ for almost all $\fp$. For such $\fp$, we can write $a_{\fp}$	with affine coordinates $a_{\fp}=(x_{\fp},y_{\fp})$. We now enlarge	$S'$ so that: 
				\begin{itemize}
					\item [(i)] $S\subseteq S'$.
					\item [(ii)] $a_{\fp}\neq O$ for $\fp\notin S'$.
					\item [(iii)] All points in $V(K)-\{O\}$ have $\calO_{K,S'}$-integral affine coordinates. \changeK{This means}
					$x_{\fp},y_{\fp}\in \calO_{\fp}$ for $\fp\notin S'$.
			  	\item [\changeK{(iv)}] \changeK{$A,B\in \calO_{K,S'}$}.
			  \end{itemize} 
			\changeK{Since $a\in \cC(\calO_{\varphi}(P))$, properties (i), (iii) and  ~\eqref{eq:ValProp}
			imply that $v(x(\varphi^n(P)))\geq 0$ for every $n$, and
			every $v\notin S'$. By property (iv) and the Weierstrass equation for $X$,
			we have $v(y(\varphi^n(P)))\geq 0$ for every $n$, and every $v\notin S'$.
			Therefore  
			$\varphi^n(P)$ has $\calO_{K,S'}$-integral affine coordinates
			for every $n$. This contradicts Siegel's theorem asserting finiteness of
			integral points.}
		\end{proof}

%%%%%%%%%%%%%%%%%%%%%%%%%%%%%%%%%%%%%%%%%%%%%%%%%%%%%%%%%%%%%%%%%%%%%%%%%%%%%%%%
%%%%%%%%%%%%%%%%%%				Bibliography			%%%%%%%%%%%%%%%%%%%%%%%%
%%%%%%%%%%%%%%%%%%%%%%%%%%%%%%%%%%%%%%%%%%%%%%%%%%%%%%%%%%%%%%%%%%%%%%%%%%%%%%%%

\end{document}